\documentclass[a4paper, 11pt]{amsart}

\usepackage{amsmath,amssymb,amsthm,enumitem,overpic,graphicx,url,overpic}
\usepackage[margin=1.2in]{geometry}
\usepackage[dvipsnames]{xcolor}
\definecolor{PrettyGreen}{RGB}{0, 180, 0}
\usepackage[all]{xy}

\edef\restoreparindent{\parindent=\the\parindent\relax}
\usepackage{parskip}
\restoreparindent

\usepackage[breaklinks,colorlinks=true,allcolors=PrettyGreen]{hyperref}

\title{A big mapping class acting parabolically on the nonseparating curve graph}
\date{}

\author{Federica Fanoni}
\email{federica.fanoni@u-pec.fr}
\address{CNRS, Univ Paris Est Creteil, Univ Gustave Eiffel, LAMAR8050, F-94010 Creteil, France}
\author{Sebastian Hensel}
\email{hensel@math.lmu.de}
\address{Mathematisches Institut der Universit\"at M\"unchen, Theresienstr. 39, 80333 M\"unchen, Germany}


\newtheorem{thm}{Theorem}[section]

\newtheorem{claim}[thm]{Claim}

\newtheorem{theorintro}{Theorem}

\theoremstyle{definition}

\newcommand{\NC}{\mathcal{NC}}
\newcommand{\NCdagger}{\NC^\dagger}
\newcommand{\ssm}{\smallsetminus}

\DeclareMathOperator{\dist}{d}
\DeclareMathOperator{\distdagger}{\dist^\dagger}

\newcommand{\Z}{\mathbb{Z}}
\newcommand{\R}{\mathbb{R}}

\newcommand{\st}{\;|\;}

\begin{document}
\begin{abstract}
We use fine curve graph tools to prove that there exist parabolic isometries of graphs of curves associated to surfaces of infinite type.
\end{abstract}

\maketitle

\section{Introduction}
A cornerstone result in low-dimensional geometry and topology is the classification by Nielsen and Thurston \cite{thurston_geometry} of mapping classes of finite-type surfaces, into periodic, reducible and pseudo-Anosov elements. A consequence of the classification is that a mapping class is pseudo-Anosov if and only if it has no periodic curve. Looking at the action on the curve graph, this means that the only elements that can act as parabolic or hyperbolic isometries are the pseudo-Anosov ones. In \cite{mm_geometry}, Masur and Minsky showed that pseudo-Anosov mapping classes always act as hyperbolic isometries. As the automorphism group of the curve graph is the (extended) mapping class group \cite{ivanov_automorphisms, korkmaz_automorphisms, luo_automorphisms}, this means that the curve graph has no parabolic isometries.

If the surface is not of finite type, the curve graph has always diameter two. On the other hand, for a large class of infinite-type surface, there are constructions of spaces that could play the role of the curve graph in this setting: infinite-diameter connected Gromov hyperbolic spaces on which the mapping class group acts with unbounded orbits (see \cite{bavard_hyperbolicite, dfv_graphs, av_geometry, rasmussen_uniform, rasmussen_geometry, hqr_big, fgm_homeomorphic, bnv_grand,  schaffercohen_graphs, kopreski_asymptotic}). For these examples, a natural question has been open since roughly a decade (see \cite[Problem 2.55]{aim} for the question stated for an example): do these spaces admit parabolic isometries? It is important to remark that in this setup only partial results about the classification of mapping classes are known (\cite{ccf_endperiodic}, \cite{bft_towards}), making the situation very different from the finite-type case.

Our main result is the existence of a parabolic isometry in the case of the nonseparating curve graph of an infinite-type surface of genus one.

\begin{theorintro}\label{thm:big-parabolic-intro}
Let \(T\) be a torus and let \(S\) be \(T\) with a Cantor set \(C\) and a point \(p_0\) removed. Then there is a homeomorphism \(F\) of \(T\), fixing \(C\) and \(p_0\), such that \(F\) is a parabolic isometry of \(\NCdagger(T)\) and \([F]\) is a parabolic isometry of \(\NC(S)\). 
\end{theorintro}


The main idea to prove Theorem \ref{thm:big-parabolic-intro} is to relate (non-)hyperbolicity of a mapping class acting on \(\NC(S)\) to (non-)hyperbolicity of a representative homeomorphism acting on \(\NCdagger(T^2)\), and then use the hyperbolicity criterion given by Bowden, Mann, Militon, Webb and the second author in \cite{bhmmw_rotation}.

\section*{Acknowledgements}
The authors are very grateful to the referee for their careful reading, and in particular for finding a mistake in a previous version of this paper.

This research was partly funded by the B\'ezout Labex (reference ANR-10-LABX-58), the Tremplin --- ERC Starting Grant MAGIC (reference  ANR-23-TERC-0007) and the ANR PRC grant GALS (reference ANR-23-CE40-0001), all funded by the Agence nationale de la recherche (ANR). The authors are grateful for their support. 

\section{Preliminaries and notation}
Unless otherwise stated, by \emph{surface} we mean a connected orientable two-manifold without boundary. A \emph{curve} on a surface will be assumed to be simple, closed and non-contractible. A curve is \emph{nonseparating} if it does not disconnect the surface.

The \emph{nonseparating fine curve graph} \(\NCdagger(S)\) of a surface of positive genus \(S\) has nonseparating curves as vertices and edges correspond to disjointness. The \emph{nonseparating curve graph} \(\NC(S)\) has homotopy classes of nonseparating curves as vertices and two classes are adjacent if they admit disjoint representatives. If $S$ is a torus, edges instead correspond to curves intersecting at most once.
We will denote by \(\distdagger\) (respectively, \(\dist\)) the distance in \(\NCdagger(S)\) (respectively, \(\NC(S)\)).

All the graphs mentioned above are Gromov hyperbolic (by work of Rasmussen \cite{rasmussen_uniform} and Bowden, Hensel and Webb \cite{bhw_quasi}), which allows us to talk about their boundary and to classify their isometries. We recall here the definitions and results related to Gromov hyperbolic spaces that we will need in the article. For more details, we refer to \cite{bh_metric} and \cite{vaisala_gromov}.

\subsection{Gromov hyperbolic spaces}

Let \(X\) be a metric space. We say that it is \emph{\(\delta\)-hyperbolic} if any side of any geodesic triangle is contained in the \(\delta\)-neighborhood of the union of the other two sides. We say that \(X\) is \emph{Gromov hyperbolic} if it is \(\delta\)-hyperbolic for some \(\delta\geq 0\). For the remainder of this section, we assume that \(X\) is a Gromov hyperbolic space.

For any isometry \(f\) of \(X\) and any point \(x\in X\) the limit \[\lim_{n\to\infty}\frac{\dist(x,f^n(x))}{n}\]
exists and is independent of \(x\). We call this limit the \emph{asymptotic translation length} of \(f\).

An isometry \(f\) of \(X\) is of one of three types:
\begin{itemize}
\item \emph{elliptic}, if it has bounded orbits;
\item \emph{parabolic}, if it has no bounded orbits, but its asymptotic translation length is zero;
\item \emph{hyperbolic}, if its asymptotic translation length is positive.
\end{itemize}

\subsection{Rotation sets}
Let \(F\) be a homotopically trivial homeomorphism of the torus \(T\simeq \R^2/\Z^2\). Choose a lift \(\tilde{F}\) of \(F\) to the universal cover \(\R^2\). The \emph{rotation set} \(\rho(\tilde{F})\) of \(\tilde{F}\) is the set of all limits of converging sequences of the form
\[\frac{\tilde{F}^{n_k}(x_k)-x_k}{n_k},\]
where \(\{x_k\}_k\) is a sequence of points in the plane and \(\{n_k\}_k\) is a sequence of integers going to infinity. By a result of Misiurewicz and Ziemian \cite{mz_rotation}, \(\rho(\tilde{F})\) coincides with the convex hull of the \emph{pointwise rotation set} of \(\tilde{F}\), which is the collection of all limit points of converging subsequences of the form
\[\frac{\tilde{F}^{n_k}(x)-x}{n_k},\]
where \(\{n_k\}_k\) is a sequence of integers going to infinity.

If we choose another lift \(\tilde{F}'\) of \(F\), the rotation sets of \(\tilde{F}\) and \(\tilde{F}'\) differ only by an integer translation. We can therefore define the \emph{rotation set} \(\rho(F)\) of \(F\) to be \(\rho(\tilde{F})\) modulo \(\Z^2\), for any lift \(\tilde{F}\) of \(F\). In particular, properties such as having empty interior or being one-dimensional are well-defined for rotation sets of homotopically trivial homeomorphisms of \(T\).

\section{Proof of Theorem \ref{thm:big-parabolic-intro}}\label{sec:parabolic}
To prove Theorem \ref{thm:big-parabolic-intro}, we will give an explicit construction of a parabolic mapping class.

Let \(f:S^1\to S^1\) be a Denjoy map, arising from a rotation of some angle \(\alpha\notin\mathbb{Q}\) by blowing up the orbit of a point. Denote by \(K\) the invariant Cantor set. Let \(T=S^1\times S^1\) be the mapping torus of \(f\) and consider the vertical vector field \(V\) and the associated flow \(\varphi_t\). The Cantor set \(C:=K\times \{0\}\) is invariant under \(\varphi_1\).

Pick a point \(p=(x_0,0)\notin C\) and \(\varepsilon>0\) sufficiently small so that the \(\varepsilon\)-neighborhood \(U\) of \(\{\varphi_t(p)\st t\in [0,1/4]\}\) is contained in the complement of \(\{\varphi_t(C)\st t\in\R\}\), and disjoint from $S^1 \times \left\{\frac12\right\}$. Consider a smooth function \(\rho:T\to\R_{\geq 0}\) such that:
\begin{itemize}
\item \(\rho\) is positive outside \(\{\varphi_t(p)\st t\in [0,1/4]\}\);
\item \(\rho\) is zero on \(\{\varphi_t(p)\st t\in [0,1/4]\}\); 
\item \(\rho\) is one outside \(U\).
\end{itemize}
Then the time-one map \(F\) of the dampened vector field \(\rho V\) is an orientation-preserving homeomorphism of \(S:=T\ssm C\ssm \{p\}\) and it is homotopic to the identity. To prove Theorem \ref{thm:big-parabolic-intro} our goal is to show:

\begin{thm}\label{thm:big-parabolic}
The map \(F\) acts parabolically on \(\NCdagger(T)\) and induces a mapping class acting parabolically on \(\NC(S)\).
\end{thm}

The theorem follows from two claims:

\begin{claim}\label{claim:rot-set}
The rotation set of \(F\) acting on \(\NCdagger(T)\) is contained in a line.
\end{claim}

\begin{claim}\label{claim:not-elliptic}
There are no bounded orbits of \([F]\) acting on \(\NC(S)\).
\end{claim}

Indeed, Claim \ref{claim:rot-set}, together with \cite[Theorem 1.3]{bhmmw_rotation}, implies that \(F\) does not act hyperbolically on \(\NCdagger(T)\) and therefore (since the projection \(\NCdagger(T)\to\NC(S)\) is distance non-increasing), \([F]\) does not act hyperbolically on \(NC(S)\). By Claim \ref{claim:not-elliptic}, \([F]\) is not an elliptic element, so it is parabolic. In particular, \(F\) is not elliptic either, and therefore it is parabolic as well.

Let us prove the claims.

\begin{proof}[Proof of Claim \ref{claim:rot-set}]
The intuition is that \(F\) preserves the flow-lines of \(V\), so the rotation set must be one-dimensional. 

More precisely, we can pick a lift \(\tilde{f}:\R\to\R\) of \(f\) such that \(\tilde{f}(0)\geq 0\) and is as small as possible. Using this, we can construct a lift \(\tilde{F}\) of \(F\) such that \(\tilde{F}(x,y)=(\tilde{f}(x),y+1)\) for every \(x\in \tilde{C}\), where \(\tilde{C}\) is the lift of the Cantor set \(C\).

Given any \((x,y)\in [0,1]^2\), its sequence of iterates is contained in a broken line of the form
\[\{\dots (\tilde{f}^{-m}(z),-m),\dots, (\tilde{f}^{-1}(z),-1), (z,0), (\tilde{f}(z),1),\dots, (\tilde{f}^{m}(z),m), \dots\}.\]
By construction of the Denjoy map, \(|\tilde{f}^m(z)-(z+m\alpha)|<1\) for every \(z\) and \(m\), so the average slope of the broken line is \(\frac{1}{\alpha}\). This implies that
\[\lim_{n\to\infty}\frac{\tilde{F}^n(x,y)-(x,y)}{n}\]
is proportional to \((\alpha,1)\). As a consequence, the rotation set of \(F\) is contained in the line \(y=\frac{x}{\alpha}\).
\end{proof}
 
\begin{proof}[Proof of Claim \ref{claim:not-elliptic}]
  Let \(\gamma\) be the curve \(S^1\times\left\{\frac{1}{2}\right\}\) and \(\gamma'\) the curve \(S^1\times\left\{\frac{3}{4}\right\}\). We will show that the orbit in \(\NC(S)\) of \([\gamma]\) under the action of \([F]\) is unbounded, which implies that every \([F]\)-orbit is unbounded.

To do so, we first study the intersection pattern
    between $\gamma'$ and $F^n(\gamma)$ for large $n$. This is easier
    to do in the universal cover $\mathbb{R}^2$ of the torus. Let
    $\widetilde{V}$ be a lift of the vector field $V$ used in the
    construction of $F$, and let $\widetilde{F}$ be a lift of
    $F$. This is time-one map of the dampened vector field
    $\rho\circ \pi \widetilde{V}$ where $\pi$ is the covering
    projection. Hence, $\widetilde{F}^n$ is the time--$n$--map of that
    vector field. Intuitively, this map has the effect of ``flowing
    along $\widetilde{V}$, but getting caught at each lift of $p$''
    --- see Figure \ref{fig:lifts}.
    
    From this description we can see that every bigon (in \(S\)) that
    \(F^n(\gamma)\) forms a bigon with the curve \(\gamma'\) contains
    the marked point $p$ or a point of $C$. In other words,
    \(F^n(\gamma)\) and \(\gamma'\) are in minimal position on $T\setminus (\{p\} \cup C)$.
    So by \cite[Lemma 3.4]{bhw_quasi} we have:
    \[\dist([\gamma],[F^n(\gamma)])=\dist([\gamma'],[F^n(\gamma)])\geq\distdagger(\gamma',F^n(\gamma))\geq \distdagger(\gamma,F^n(\gamma))-1.\]

%

\begin{figure}[h]
\begin{overpic}[scale=2]{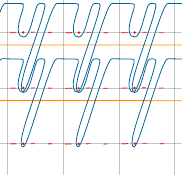}
\put(102, 40){\color[RGB]{255, 140, 0}\(\gamma'\)}
\put(102, 70){\color[RGB]{255, 140, 0}\(\gamma'\)}
\put(-17, 90){\color[RGB]{0, 99, 150}\(F^2(\gamma)\)}
\put(-17, 60){\color[RGB]{0, 99, 150}\(F^2(\gamma)\)}

\end{overpic}
\caption{Lifts of \(\gamma'\) (in orange) and of \(F^2(\gamma)\) (in blue) to the universal cover of \(T\). The lifts of \(p\) and \(C\) are in red.}\label{fig:lifts}
\end{figure}

If by contradiction the \([F]\)-orbit of \([\gamma]\) is bounded, there is some \(K\geq 0\) so that for every \(n\)
\[\distdagger(\gamma,F^n(\gamma)) \leq \dist([\gamma],[F^n(\gamma)])+1\leq K.\]

Fix a genus-two branched cover \(S_2\to T\), branched over a point in the Cantor set; then the elevations of \(\gamma\) and of \(F^n(\gamma)\) also have bounded distance if \(\NCdagger(S_2)\). By \cite{bhmmw_rotation}, there is a finite-sheeted cover \(X\to S_2\) and elevations of the elevations of \(\gamma\) and \(F^n(\gamma)\) which are disjoint.

Let \(\tilde{\gamma_1},\dots,\tilde{\gamma_N}\subset X\) be the elevations of \(\gamma\). Let \(\lambda\subset T\) be the flow lines of all points in \(C\):
\[\lambda:=\{\varphi_t(C)\st t\in\R\}.\]
Denote by \(\lambda_2\) the lift of \(\lambda\) in \(S_2\) and \(\lambda_X\) the lift to \(X\). Fix a metric on \(X\). We will show that any sufficiently long segment \(L\) of \(\lambda_X\) is so that, for every \(i\), \(\tilde{\gamma_i}\cup L\) intersects all nonseparating curves of \(X\).

As a consequence, if \(\{\beta_n\}_n\subset\NCdagger(T)\) is a sequence of curves which are disjoint from leaf segments of \(\lambda\) of length going to infinity, \(\distdagger(\beta_n,\gamma)\to\infty\). As the \(F^n(\gamma)\) satisfy the condition of being disjoint from such leaf segments, this concludes the proof of Claim \ref{claim:not-elliptic}.

To prove our statement, consider the map \(c:T\to S^1\times S^1=:\hat{T}\) which collapses the bi-infinite strips obtained by flowing the complementary segments of \(C\). The covering maps \(X\to S_2\to T\) induce covering maps \(\hat{X}\to\hat{S}_2\to\hat{T}\), which commute with analogous collapsing maps \(c_2:S_2\to\hat{S}_2\) and \(c_X:X\to\hat{X}\):
\[
\xymatrix{
\lambda_X,\tilde{\gamma_1},\dots,\tilde{\gamma}_N\subset\hspace{-1cm}& X \ar[r]^{c_X}\ar[d] & \hat{X} \ar[d]&\hspace{-1.2cm}\supset \hat{\lambda}_X\\
\lambda_2\subset\hspace{-3cm}&S_2 \ar[r]^{c_2}\ar[d] & \hat{S}_2 \ar[d]&\hspace{-1.2cm}\supset \hat{\lambda}_2\\
\lambda,\gamma\subset\hspace{-2.8cm} &T \ar[r]^{c} & \hat{T}&\hspace{-1.4cm}\supset \hat{\lambda}
}
\]
We think of the surfaces \(\hat{S}_2\) and \(\hat{X}\) as square-tiled surfaces with the tiling coming from the projections to \(\hat{T}\). Define \(\hat{\lambda}:=c(\lambda)\), \(\hat{\lambda}_2:=c_2(\lambda_2)\) and \(\hat{\lambda}_X:=c_X(\lambda_X)\). Then \(\hat{\lambda}\) is an irrational slope foliation of the torus, and \(\hat{\lambda}_2\) and \(\hat{\lambda}_X\) are irrational slope foliations of square-tiled surfaces.

Set \(\hat{\gamma_i}:=c_X(\tilde{\gamma_i})\). Since \(\hat{\lambda}_X\) is an irrational slope foliation, there is some length \(\hat{\ell}\) so that for every segment \(\hat{L}\) of \(\hat{\lambda_X}\) of length at least \(\ell\), the complementary components are \(\hat{L}\cup\hat{\gamma_i}\) rectangles horizontal sides and sides with irrational slope, possibly with one or two irrational slope slits (if \(\hat{L}\) does not start or end on the curve). Let \(\ell\) be such that every segment of \(\lambda_X\) of length at least \(\ell\) projects onto a segment of \(\hat{\lambda_X}\) of length at least \(\hat{\ell}\). Fix a segment \(L\subset \lambda_X\) of length at least \(\ell\). Suppose by contradiction that \(\beta\) is a nonseparating curve disjoint from \(L\cup\tilde{\gamma_i}\), i.e.\ \(\beta\) is contained in a connected component of \(X\ssm (L\cup\tilde{\gamma_i})\). Such a component comes from a component of \(\hat{X}\ssm (\hat{L}\cup\hat{\gamma_i})\), possibly blown-up along some irrational slope segment. In particular, any such component is contractible and cannot contain a nonseparating curve, a contradiction.
\end{proof}

\bibliographystyle{plain}
\bibliography{bibliography}

\end{document}